\documentclass[11pt]{amsart}
\usepackage{amsmath,amsthm, amsfonts, amssymb, latexsym,verbatim,bbm,longtable}
\input xy
\xyoption{all}
\usepackage{hyperref,multirow}
\usepackage{tikz, ifthen}
\usepackage[shortlabels]{enumitem}
\usepackage{young}
\usepackage{mathtools}
\usepackage{dynkin-diagrams}
\usepackage{standalone}
\usepackage{multicol}
\usepackage{caption}
\usepackage{subcaption}
\usepackage{float}
\usepackage{dynkin-diagrams}

\allowdisplaybreaks[2] \textwidth15.1cm \textheight22cm \headheight12pt \oddsidemargin.4cm
\theoremstyle{plain}
\newtheorem{theorem}{Theorem}[section]

\newtheorem{lemma}[theorem]{Lemma}
\newtheorem{proposition}[theorem]{Proposition}
\newtheorem{prop}[theorem]{Proposition}

\newtheorem{definition}[theorem]{Definition}

\newtheorem{corollary}[theorem]{Corollary}

\newtheorem{question}[theorem]{Question}

\newcommand{\Z}{\mathbb{Z}}

\newcommand{\newword}{\textbf}

\DeclareMathOperator{\cs}{C}

\DeclareMathOperator{\unfold}{unfold}
\DeclareMathOperator{\fold}{fold}

\begin{document}

\title{Demazure product of permutations and hopping}

\author{Tina Li}
\address{Clements High School, Sugar Land, TX}
%\email{tinali.jyq@gmail.com}

\author{Suho Oh}
\address{Department of Mathematics, Texas State University, San Marcos, TX}
\email{suhooh@txstate.edu}

\author{Edward Richmond}
\address{Department of Mathematics, Oklahoma State University, Stillwater, OK}
\email{edward.richmond@okstate.edu}

\author{Grace Yan}
\address{Morgantown High School, Morgantown, WV}
%\email{graceyan212@gmail.com}

\author{Kimberley You}
\address{Pioneer High School, Mission, TX}
%\email{kimberlyjingyuanyou@gmail.com}

\begin{abstract}
The Demazure product (also goes by the name of 0-Hecke product or the greedy product) is an associative operation on Coxeter groups with interesting properties and important applications. In this note, we study permutations and present an efficient way to compute the Demazure product of two permutations starting from their usual product and then applying a new operator we call a hopping operator.  We also give an analogous result for the group of signed permutations.
\end{abstract}

\maketitle

\section{Introduction}\label{section:intro}

Coxeter groups play an important role in the representation theory of Lie groups and the geometry of associated flag varieties and Schubert varieties. There is an interesting associative operation on Coxeter groups called the \newword{Demazure product} \cite{Dem94} (also called the $0$-\newword{Hecke product} or the \newword{greedy product}).  In this paper, we study the Demazure product for two classes of Coxeter groups: permutations and signed permutations.  Our main result is a computationally efficient algorithm to compute this product using a new operator called a hopping operator.

Let $W$ be a Coxeter group with simple generating set $S$.  Specifically, $W$ is generated by $S$ with relations of the the form
\begin{equation}\label{eqn:Coxeter_braid}
(st)^{m_{st}}=id, \quad s,t\in S
\end{equation}
for some values $m_{st}\in \Z_{>0}\cup\{\infty\}$ where $m_{st}=1$ if and only if $s=t$.
Coxeter groups come equipped with a natural length function $\ell:W\rightarrow \Z_{\geq 0}$ and a poset structure called the Bruhat order (denoted by $\leq$) that respects length.  For more details on the basic properties of Coxeter groups, see \cite{Bjorner-Brenti05}.  The \newword{Coxeter monoid} structure (also called the $0$-\newword{Iwahari-Hecke monoid}) on $W$ is defined to be the monoid generated by $S$ with a product $\star$ satisfying the Coxeter braid relations in Equation \eqref{eqn:Coxeter_braid} for $s\neq t$ along with the relation $s\star s=s$ for all $s\in S$ (this new relation replaces $s^2=id$ in the usual product).  This monoid structure was first studied by Norton in \cite{No79} in the context of Hecke algebras.  It is well known that, as sets, $W=\langle S,\star\rangle$.  We say an expression $w=s_1\cdots s_k$ is reduced if $\ell(w)=k$.  In other words, $w$ cannot be expressed with fewer than $k$ generators in $S$.  The next lemma records some basic facts about the Coxeter monoid.
\begin{lemma}\label{lemma:moniod_properties}
Let $W$ be a Coxeter group with simple generating set $S$.  Then the following are true:
\begin{enumerate}
    \item Let $(s_1,\ldots,s_k)$ be a sequence of generators in $S$.  Then
    $$s_1\cdots s_k\leq s_1\star\cdots \star s_k$$ with equality if and only if $(s_1,\ldots,s_k)$ is a reduced expression.
    %\item If $w'\leq w$ and $u'\leq u$, then $w'u'\leq w'\star u'\leq w\star u$.
    \item For any $s\in S$ and $w\in W$,
$$s\star w=\begin{cases} w & \text{if}\quad \ell(sw)< \ell(w)\\ sw & \text{if}\quad \ell(sw)>\ell(w).\end{cases}$$
\end{enumerate}
\end{lemma}

It turns out that there is a very nice interpretation of $w \star u$:

%The next proposition is a key property about the monoid structure of $W$ in relation to Bruhat intervals.  It is proved by He in \cite[Lemma 1]{He09} and later by Kenny in \cite[Proposition 8]{Kenny14}.

\begin{prop}\cite{He09, Kenny14}\label{prop:interval_product}
For any $w,u\in W$, the Bruhat interval $$[e,w\star u]=\{ab \ |\ a\in [e,w], b\in[e,u]\}.$$
\end{prop}

We give an example of this phenomenon. The poset in Figure~\ref{fig:4231lower} is the Bruhat order of the symmetric group $S_4$ which has three simple generators $s_1,s_2,s_3$. The elements in the lower interval of $s_1s_2s_3s_2s_1$ are colored in red. All the elements in the lower interval $[e,s_1s_2s_3s_2s_1]$ can be written as $a \star b$ where $a \leq s_1s_2s_3$ and $b \leq s_2s_1$. For example, $s_2s_3s_1$ can be written as $s_2s_3 \star s_1$.

\begin{figure}[h]

 \begin{center}

  \vspace{3.7mm}
\begin{tikzpicture}[scale=0.7]
 \node (4321) at (0,10) {$s_1s_2s_3s_1s_2s_1$};

 \node (4312) at (5,8) {$s_2s_3s_2s_1s_2$};
 \node[black!20!red] (4231) at (-5,8) {$s_1s_2s_3s_2s_1$};
 \node (3421) at (0,8) {$s_1s_2s_3s_1s_2$};

 \node[black!20!red] (4132) at (0,6) {$s_2s_3s_2s_1$};
 \node[black!20!red] (4213) at (8,6) {$s_3s_1s_2s_1$};
 \node (3412) at (4,6) {$s_2s_3s_1s_2$};
 \node[black!20!red] (2431) at (-8,6) {$s_1s_2s_3s_2$};
 \node[black!20!red] (3241) at (-4,6) {$s_1s_2s_3s_1$};

 \node[black!20!red] (1432) at (-6,4) {$s_2s_3s_2$};
 \node[black!20!red] (4123) at (2,4) {$s_3s_2s_1$};
 \node[black!20!red] (2413) at (6,4) {$s_3s_1s_2$};
 \node[black!20!red] (3142) at (-2,4) {$s_2s_3s_1$};
 \node[black!20!red] (3214) at (10,4) {$s_1s_2s_1$};
 \node[black!20!red] (2341) at (-10,4) {$s_1s_2s_3$};

 \node[black!20!red] (1423) at (-4,2) {$s_3s_2$};
 \node[black!20!red] (1342) at (-8,2) {$s_2s_3$};
  \node[black!20!red] (2143) at (0,2) {$s_3s_1$};
 \node[black!20!red] (3124) at (4,2) {$s_2s_1$};
 \node[black!20!red] (2314) at (8,2) {$s_1s_2$};

 \node[black!20!red] (1243) at (-5,0) {$s_3$};
 \node[black!20!red] (1324) at (0,0) {$s_2$};
 \node[black!20!red] (2134) at (5,0) {$s_1$};

  \node[black!20!red] (1234) at (0,-2) {$\emptyset$};

\draw (1234) -- (1243);
\draw (1234) -- (1324);
\draw (1234) -- (2134);
\draw (1243) -- (1342);
\draw (1243) -- (1423);
\draw (1243) -- (2143);
\draw (1324) -- (1342);
\draw (1324) -- (1423);
\draw (1324) -- (2314);
\draw (1324) -- (3124);
\draw (2134) -- (2143);
\draw (2134) -- (2314);
\draw (2134) -- (3124);
\draw (1342) -- (1432);
\draw (1342) -- (2341);
\draw (1342) -- (3142);
\draw (1423) -- (1432);
\draw (1423) -- (2413);
\draw (1423) -- (4123);
\draw (2143) -- (2341);
\draw (2143) -- (2413);
\draw (2143) -- (3142);
\draw (2143) -- (4123);
\draw (2314) -- (2341);
\draw (2314) -- (2413);
\draw (2314) -- (3214);
\draw (3124) -- (3142);
\draw (3124) -- (3214);
\draw (3124) -- (4123);

\draw (1432) -- (2431);
\draw (1432) -- (3412);
\draw (1432) -- (4132);

\draw (2341) -- (2431);
\draw (2341) -- (3241);

\draw (2413) -- (2431);
\draw (2413) -- (4213);

\draw (3142) -- (3241);
\draw (3142) -- (4132);
\draw (3142) -- (3412);

\draw (3214) -- (3241);
\draw (3214) -- (3412);
\draw (3214) -- (4213);

\draw (4123) -- (4132);
\draw (4123) -- (4213);

\draw (2431) -- (3421);
\draw (2431) -- (4231);

\draw (3241) -- (3421);
\draw (3241) -- (4231);

\draw (3412) -- (3421);
\draw (3412) -- (4312);

\draw (4132) -- (4231);
\draw (4132) -- (4312);

\draw (4213) -- (4231);
\draw (4213) -- (4312);

\draw (3421) -- (4321);
\draw (4231) -- (4321);
\draw (4312) -- (4321);

\end{tikzpicture}
    \captionsetup{width=1.0\linewidth}
  \captionof{figure}{The Bruhat order of $S_4$ and the lower interval of $s_1s_2s_3s_2s_1$ in red.}
  \label{fig:4231lower}

 \end{center}

\end{figure}
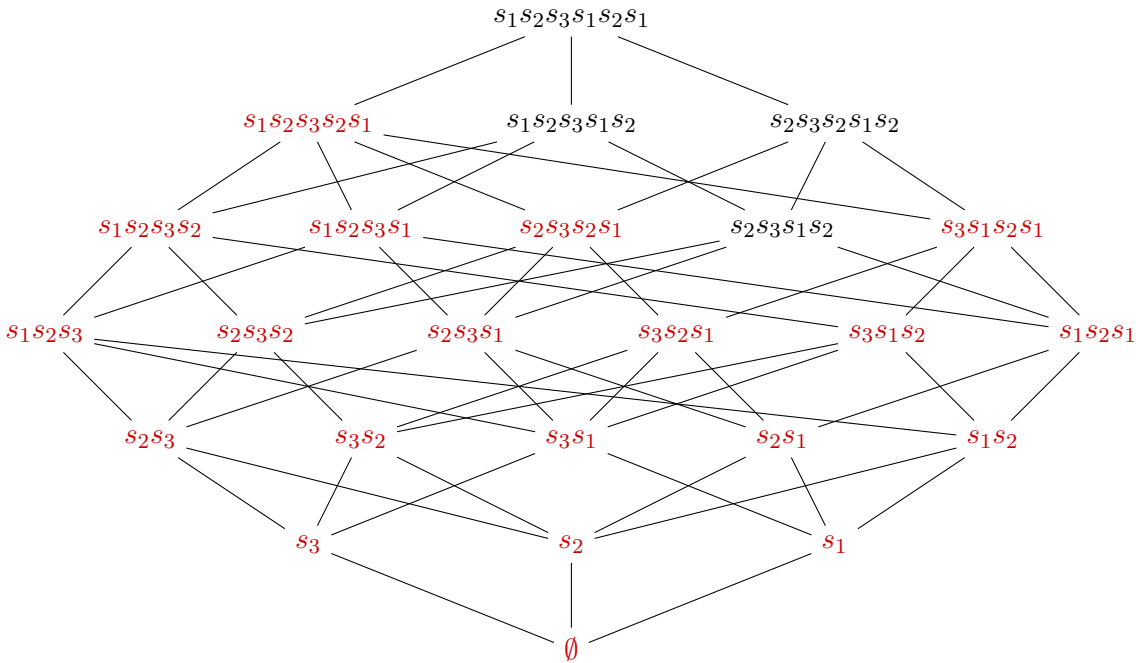

This product has been used and studied in various fields that depend on Coxeter groups \cite{gay:hal-01691266}, \cite{KNUTSON2004161},\cite{Larson19}, \cite{OR}, \cite{Pflueger}, \cite{Richardson-Springer90}.  For example, in Lie theory, the Demazure product naturally arises in the study of BN pairs and reductive groups.  Specifically, the relation on Borel double orbit closures is given by $$\overline{BwBuB}=\overline{B(w\star u) B}.$$ While the product $w\star v$ has been well studied for many years, it is computationally expensive to calculate using reduced expressions of $w$ and $u$.  In this note, we present an efficient method to compute the Demazure product of two permutations in the symmetric group using their one-line notation.  This algorithm starts with the usual product of permutations and a brand new operation we call \newword{hopping} operators.  We state this result in Theorem \ref{thm:main}. In Section \ref{s:typeB}, we prove an analogous result for signed permutations which is stated in Theorem \ref{thm:b_main}.

%For any simple transposition $s_i$ and $w \in S_d$, we define the Demazure product $s_i \star w$ as:

%\[
%    s_i \star w=
%\begin{cases}
%    s_i & \text{if } \ell(s_iw) > \ell(w) \\
%    w,              & \text{otherwise.}
%\end{cases}
%\]

%Now extending this via the associative property

\section{The hopping operator}

In this section, we focus on the permutation group or symmetric group $S_n$.  These groups are also known as Coxeter groups of type $A$.   The group $S_n$ is a Coxeter group with simple generating set $S=\{s_1,\ldots,s_{n-1}\}$ satisfying the relations $s_i^2=id$ and
\begin{equation}\label{eqn:type_A_braid}
(s_is_j)^2=id\ \text{ if $|i-j|>1$\ and}\ (s_is_{i+1})^3=id.
\end{equation}
The generator $s_i$ corresponds to the simple transposition $(i,i+1)$.  Let $[n]:=\{1,2,\ldots, n\}$ and for $w\in S_n$, let $w=w(1)w(2)\cdots w(n)$ denote the permutation in one-line notation.  We define a new operator on permutations called the \newword{hopping operator}.

\begin{definition}\label{def:hopping_typeA}
For $t \in [n]$ and $L$ an ordered subset of $[n]$ (without repetition) the \newword{hopping operator} $$h_{t,L}:S_n\rightarrow S_n$$ acts on a permutation $w$ to yield the permutation obtained by the following algorithm: Scan to the right (within the one-line notation of $w$) of $t$ and look for the element furthest to the right in $L$ that is greater than $t$. If it exists, swap $t$ and that element, replace $w$ with the resulting permutation, and repeat. The algorithm ends when there are no elements of $L$ within $w$ to the right of $t$.
\end{definition}

For example, take $w = 891726435$. Then $h_{1,[2,3,4,5,6,7,8]}(w) = 897625431$ is obtained by the following process:
$$89\textbf{17}26435 \rightarrow 897\textbf{1}2\textbf{6}435 \rightarrow 89762\textbf{1}43\textbf{5} \rightarrow 897625431.$$
For another example, we have $h_{1,[3,6,5,7,2]}(w) = 892756431$ and is obtained by the following process:
$$89\textbf{1}7\textbf{2}6435 \rightarrow 8927\textbf{1}643\textbf{5} \rightarrow 892756431.$$

For any ordered subset $L\subset [n]$, let $w(L)\subseteq [n]$ denote the ordered list obtained by $w$ acting on the elements of $L$.  While a permutation may not preserve $L$, it can be viewed as an operator on ordered subsets of [n] preserving size.  Hopping operators satisfy the following commuting relation with simple transpositions:

\begin{lemma}
\label{lem:transhop}
Let $w\in S_n$.  For any $i \geq t$ and ordered subset $L\subseteq [n]$, we have $$s_i\cdot h_{t,L}(w) = h_{t,s_i(L)}( s_i\cdot w).$$
\end{lemma}
\begin{proof}
The positions that $t$ will be in throughout the hopping of $h_{t,L}\,w$ will be exactly the same as the position of $t$ throughout the hopping of $h_{t,s_i(L)}(s_iw)$. So $h_{t,L}w$ and $h_{t,s_i(L)}(s_iw)$ will be exactly the same except $i$ and $i+1$ flipped, hence the desired result.
\end{proof}

For example, let $w = 514632$. We compare the action of $s_3 h_{1,[2,3,4,5]}$ and $h_{1,[2,4,3,5]}s_3$ on $w$. First, we have
$$ 514632 \xrightarrow[h_{1,[2,3,4,5]}]{} 543621 \xrightarrow[\quad s_3 \quad]{} 534621.$$
On the other hand, we get
$$514632 \xrightarrow[\quad s_3 \quad]{} 513642 \xrightarrow[h_{1,[2,4,3,5]}]{} 534621,$$
yielding the same result as guaranteed by Lemma~\ref{lem:transhop}.

%We have $h_{1,[2,3,4,5]} 514632 = 543621$ and then $s_3 543621 = 534621$. On the other hand we get $h_{1,[2,4,3,5]} s_3 514632 = 534621$, yielding the same result as guaranteed by Lemma~\ref{lem:transhop}.

Throughout the paper, we will denote the product of simple transpositions by the symbol
$$\cs_{a,b}:=s_a s_{a+1} \cdots s_{a+b-1}.$$  If $b=0$, then $\cs_{a,0}$ is the identity.  The operator $\cs_{a,b}$ acts on a permutation $w$ by mapping each of $a,a+1,\ldots,a+b$ to $a+1,\ldots,a+b,a$ respectively. In other words, it is a cyclic shift of the elements $a,a+1,\ldots,a+b$ by one. For example, in $S_8$, we have $\cs_{2,4} = s_2s_3s_4s_5$  corresponding to the permutation $13456278$.  From the above lemma we immediately get the following corollary:

\begin{corollary}
\label{cor:shcomm}
Let $w\in S_n$.  For any $a \geq t$ and ordered subset $L\subseteq [n]$, we have $$\cs_{a,b} h_{t,L}(w) = h_{t,\cs_{a,b}(L)} \cs_{a,b}(w).$$
\end{corollary}

The Demazure product with $\cs_{a,b}$ can be described using the usual product on permutations and a hopping operator as follows:

\begin{proposition}
\label{prop:cyclehop}
$(s_i s_{i+1} \cdots s_j) \star v = h_{i,[i+1,\ldots,j+1]} (s_i s_{i+1} \cdots s_j v)$
\end{proposition}

\begin{proof}
We do strong induction on $j-i$. When $j-i=0$, the claim is that $s_i \star v = h_{i,[i+1]} (s_iv)$ which is straightforward to verify. Assume for sake of induction that we have $(s_{i+1} \cdots s_j) \star v = h_{i+1,[i+2,\ldots,j+1]} (s_{i+1} \cdots s_j v)$. Then all we need to do is show that $h_{i,[i+1]} (s_i h_{i+1,[i+2,\ldots,j+1]} (s_iw)) = h_{i,[i+1,\ldots,j+1]}(w)$, by setting $w = s_i s_{i+1} \cdots s_j v$.

First note that for any permutation $w$, if $i+1$ is not encountered during the hopping process of $i$ in  $h_{i,[i+1,\ldots,j+1]}(w)$, then the hopping of $i+1$ in $h_{i+1,[i+2,\ldots,j+1]} (s_iw)$ follows the same hopping steps as with $h_{i,[i+1,\ldots,j+1]}(w)$ (i.e, the same numbers are getting swapped in the same order).  Hence $$s_i h_{i+1,[i+2,\ldots,j+1]} (s_iw) = h_{i,[i+1,\ldots,j+1]}(w).$$ Since $i$ is to the right of $i+1$ in $h_{i,[i+1,\ldots,j+1]}(w)$, by Defintion \ref{def:hopping_typeA}, we have $$h_{i,[i+1]} h_{i,[i+1,\ldots,j+1]}(w) = h_{i,[i+1,\ldots,j+1]}(w)$$ giving us the desired result.  Otherwise, if $i+1$ is encountered during the hopping process of $i$ in $h_{i,[i+1,\ldots,j+1]}(w)$, then none of the numbers $i+2,\ldots,j+1$ appear to the right of $i+1$ in $w$ (since the elements $i+2,\ldots,j+1$ would have been prioritized to be chosen for the swap). So in $h_{i,[i+1,\ldots,j+1]}(w)$, we have that none of $i+2,\ldots,j+1$ appears between $i+1$ and $i$ (with $i$ being right of $i+1$). Now the hopping of $i+1$ in $h_{i+1,[i+2,\ldots,j+1]} (s_iw)$ follows the same hopping steps except the last one of $h_{i,[i+1,\ldots,j+1]}(w)$.  Hence, we end up with $h_{i,[i+1,\ldots,j+1]}(w)$. So $$h_{i,[i+1]} (s_i h_{i+1,[i+2,\ldots,j+1]}(s_iw)) = h_{i+1,[i+2,\ldots,j+1]} (s_iw) = h_{i,[i+1,\ldots,j+1]}(w)$$ which proves the proposition.
\end{proof}

For example, let $w = 124567893 = s_3s_4s_5s_6s_7s_8=\cs_{3,6}$ and $v = 891726435$. The above proposition implies that $w \star v = h_{3,[4,5,6,7,8,9]} (wv)$. Starting with the usual product $wv = 931827546$, we get

$$9\textbf{3}1\textbf{8}27546 \rightarrow 981\textbf{3}2\textbf{7}546 \rightarrow 98172\textbf{3}54\textbf{6} \rightarrow 981726543.$$
Thus $w \star v=981726543$.

\begin{definition}\label{def:typeA_left_subword}
Given $w \in S_n$, we let $w \nwarrow a$ stand for the subword of $w$ obtained by restricting ourselves to the subword strictly left of $a$, then cutting off the elements smaller than $a$.
\end{definition}

For example, take $w = 891726435$. We have $w \nwarrow 2 = 897$, obtained by taking the subword strictly left of $1$ and then removing elements smaller than $2$. Similarly, we get $w \nwarrow 4 = 8976$, again obtained by taking the subword strictly left of $4$ and then removing elements smaller than $4$.

Notice that if $w = s_i s_{i+1} \cdots s_j$, then $ w_{\nwarrow i}=[i+1,i+2,\ldots,j+1]$. Proposition~\ref{prop:cyclehop} implies that if $w = s_i s_{i+1} \cdots s_j$, then $w \star v = h_{i,w{\nwarrow i}} (wv) $.

\section{The main result}\label{s:main_result_typeA}

In this section we give a formula for the Demazure product between two arbitrary permutations by writing one of the permutations as a product $\cs_{i,j}$'s and then carefully iterating Proposition~\ref{prop:cyclehop}.

%At the end of the previous section, we have seen that when $w = C_{i,a}$, then $w \star v = h_{i, w \nwarrow i}(wv)$. That is, in order to compute the Demazure operator all we have to do is take the usual product then simply %hop a single element $i$ according to how $w \nwarrow i$ looks like. Based on this idea, we extend it to the case when $w$ is an arbitrary permutation.

\begin{theorem}\label{thm:main}
For any $w,v\in S_n$, we have $w \star v = h_{n-1,w_{\nwarrow n-1}} h_{2,w_{\nwarrow 2}} h_{1,w_{\nwarrow 1}}(wv)$
\end{theorem}
\begin{proof}
Let $(j_1,\ldots, j_{n-1})$ denote the inversion sequence of $w$ (see \cite[Chapter 1.3]{Stanley12}).  In other words, $j_i$ denotes the number of inversions in $w$ of the form $(k,i)$.  It is easy to check that
\begin{equation}\label{eqn:type_A_stringcomposition}
w=\cs_{n-1,j_{n-1}} \cdots\cs_{2,j_2} \cs_{1,j_1}.
\end{equation}
The reason we use this decomposition is that to the left of $\cs_{i,j_i}$ restricted to $i,i+1,\ldots,n$ is exactly the same as the subword of $w$ restricted to $i,i+1,\ldots,n$.

%acts on $i,i+1,\ldots,n$ the exact same way $w$ does on those elements.

From Proposition~\ref{prop:cyclehop}, we get that
$$w \star v =  \cdots h_{i,[i+1,\ldots,j_i+1]} \cs_{i,j_i} \cdots h_{2,[3,\ldots,j_2+1]} \cs_{2,j_2} (h_{1,[2,\ldots,j_{1}+1]} \cs_{1,j_1} v).$$

It is enough to show that for each $i<n$, we have
$$\cs_{n-1,j_{n-1}} \cdots \cs_{i,j_i} h_{i,[i+1,\ldots,j_i+1]} = h_{i,w_{\nwarrow i}} \cs_{n-1,j_{n-1}} \cdots \cs_{i,j_i}.$$

This follows from Corollary~\ref{cor:shcomm} and the previous observation that $\cs_{n-1,j_{n-1}} \cdots \cs_{i,j_i}$ restricted to $i,i+1,\ldots,n$ is exactly same as that of $w$ and that we can truncate the list in a hopping operator of $h_{i,L}$ by removing the elements in $L$ that are smaller than $i$.

\end{proof}

For example, let $w=6541723$ and $v=5436217$.  The usual product of these two permutations is $wv=7142563$.  The Demazure product $w\star v$ corresponds to the sequence of hopping operators
$$h_{5,[6]}h_{4,[6,5]}h_{3,[6,5,4,7]}h_{2,[6,5,4,7]}h_{1,[6,5,4]}$$
acting on the usual product $wv$.  Applying each of the hopping operators in order to $wv$, we get %(the number on top of the arrows show what element we are hopping)
\begin{align*}
7142563 \xrightarrow[h_{1,[6,5,4]}]{}  7452613 \xrightarrow[h_{2,[6,5,4,7]}]{} & 7456213\\
& \xrightarrow[h_{3,[6,5,4,7]}]{}  7456213  \xrightarrow[h_{4,[6,5]}]{}  7564213 \xrightarrow[h_{5,[6]}]{} 7654213.
\end{align*}

%us use the above theorem to calculate $6541723 \star 5436217$. We have the usual product of these two permutations to be $7142563$.

%We need to apply the hopping operators $h_{5,[6]}h_{4,[6,5]}h_{3,[6,5,4,7]}h_{2,[6,5,4,7]}h_{1,[6,5,4]}$ to this. Applying each of the hopping operators in order to v, we get (the number on top of the arrows show what element we are hopping)
%$$7142563 \rightarrow^1 7452613 \rightarrow^2 7456213 \rightarrow^3 7456213 \rightarrow^4 7564213 \rightarrow^5 7654213.$$

This gives us $6541723 \star 5436217 = 7654213$.

\section{Signed permutations}\label{s:typeB}

In this section, we prove an analogue of Theorem \ref{thm:main} for the group of signed permutations, also known as Coxeter groups of type $B$ (or equivalently, type $C$).  Signed permutations can be viewed as a permutation subgroup of $S_{2n}$. Let $\{s'_1,\ldots,s'_{2n-1}\}$ be the simple generators of the permutation group $S_{2n}$.  We define $B_n$ to be the subgroup of $S_{2n}$ generated by $S:=\{s_1,\ldots,s_n\}$ where
\begin{equation}\label{eqn:BtoA}
s_i:=s'_i\, s'_{2n-i}\ \text{for $1\leq i< n$ and}\quad s_n:=s'_n.
\end{equation}
The convention we use regarding type $B$ simple transpositions follows those found in \cite{billey2000singular}.  As a Coxeter group, the generators $s_1,\ldots,s_{n-1}$ of $B_n$ satisfy the same relations as in type $A$ (see Equation \eqref{eqn:type_A_braid}) with the last generator $s_n$ satisfying:
$$(s_is_n)^2=id\ \text{ for $1\leq i<n-1$ and } (s_{n-1}s_n)^4=id.$$

%The Dynkin diagram of $B_7$ is drawn in Figure~\ref{fig:B7}.

%\begin{figure}[h]
%\dynkin[label,label macro/.code={s_{\drlap{#1}}},edge
%length=.75cm, scale={2}]B7
%\caption{Dynkin diagram of $B_7$.}
%\label{fig:B7}
%\end{figure}

Similar to the symmetric group, the elements of the Coxeter group $B_n$ can be interpreted using a one-line notation called the signed permutations \cite{Bjorner-Brenti05}. The convention we use here will be slightly different from that of \cite{Bjorner-Brenti05} in the sense that $s_n$ plays the role of $s_0$ in \cite{Bjorner-Brenti05}.

\begin{definition}
A \newword{signed permutation} of type $B_n$ is a permutation of $[n]$ along with a sign of $+$ or $-$ attached to each number.
\end{definition}

For example, the signed permutation $[4,-2,3,-1]$ is an element of $B_4$.  The generator $s_i\in B_n$ corresponds to the simple transposition swapping $i$ and $i+1$ if $i<n$ and $s_n$ to the transposition swapping $n$ with $-n$.   The product structure on signed permutations is just the usual composition of permutations with the added condition that $w(-i)=-w(i)$.  Let $\pm [n]$ denote the set $[n] \cup -[n]$, where $-[n] := \{-1,\ldots,-n\}$.   We impose the total ordering on $\pm [n]$ given by:
$$1 < 2<\dots < n < -n <  \dots<-2 < -1.$$
By \newword{unfolding} of a signed permutation $w \in B_n$ we mean the following: to the right of $w$, attach a reverse ordered copy of $w$ with the signs flipped to get a permutation of $\pm [n]$.  The unfolding map respects the embedding of $B_n$ as a subgroup of $S_{2n}$ given above. Specifically, if we replace $-[n]$ with $\{n+1,\ldots,2n\}$, then the unfolding map assigns to each signed permutation in $B_n$ a standard permutation in $S_{2n}$.  For example the unfolding of $[4,-2,3,-1]$ is
$$[4,-2,3,-1,1,-3,2,-4]$$ and the corresponding permutation of $[8]$ is $[4,7,3,8,1,6,2,5].$  Conversely, given a permutation of of $\pm[n]$ where the $i$-th entry is the opposite sign of $(2n+1-i)$-th entry, we can \newword{fold} the permutation to get a signed permutation on $[n]$.  In this section, we will slightly abuse notation and identify a signed permutation of $B_n$ with its unfolding in $S_{2n}$.  When referring the generators of $S_{2n}$, we set $$s'_{-i}:=s'_{2n-i}$$ for any $i< n$ and hence $s_i:=s'_is'_{-i}$.

%The \newword{folding} of a permutation of $\pm [n]$ is the reverse process of %this to get a signed permutation of $B_n$.

 %Its unfolding is $[4,-2,3,-1,1,-3,2,-4]$. We will identify $[4,-2,3,-1]$ with %$[4,-2,3,-1,1,-3,2,-4]$. Notice also that we can normalize the unfolding to get %a permutation of $[8]$: it will be $[4,7,3,8,1,6,2,5]$, coming from the ordering %$1 < \dots < 4 < -4 < \dots < -1$ we have on $\pm [4]$.

\begin{lemma}
\label{lem:b_unfold}
For any signed permutations $w,v \in B_n$, we have
\begin{equation}\label{eqn:fold-unfold}
w \star v = \fold(\unfold(w) \star \unfold(v)).
\end{equation}
\end{lemma}
\begin{proof}
First observe that Equation \eqref{eqn:fold-unfold} holds if we replace $\star$ with the group product since the unfolding map corresponds to the embedding of $B_n$ into $S_{2n}$.   We proceed by induction on the length of $w$ and will use $\ell_B,\ell_A$ to denote length in the Coxeter groups $B_n$ and $S_{2n}$ respectively.  First, suppose that $w=s_i$ where $i<n$.  Then $\unfold(w)=s'_i\, s'_{-i}$ with respect to the embedding of $B_n$ into $S_{2n}$.  Since $s'_i$ commutes with $s'_{-i}$, we have that $\ell_B(sv)=\ell_B(v)-1$ if and only if $\ell_A(\unfold(sw))=\ell_A(\unfold(w))-2$.  Lemma \ref{lemma:moniod_properties} part (2) implies
\begin{equation*}  %\label{eqn:length1}
s\star w=\fold(\unfold(s\star w))=\fold(\unfold(s)\star\unfold(w)).
\end{equation*}
A similar argument holds when $w=s_n$ and $\unfold(w)=s'_{n}$.  This proves the lemma in the case when $\ell_B(w)=1$.  Now suppose that $\ell_B(w)>1$ and write $w=sw'$ for some $s\in S$ and $w'\in B_n$ where $\ell_B(w)=\ell_B(w')+1$.  By induction we get
$$w\star v=sw'\star v=s\star (w'\star v)=s\star \fold(\unfold(w')\star \unfold(v)).$$
The inductive base case above implies
\begin{align*}
s\star \fold(\unfold(w')\star \unfold(v))&=\fold(\unfold(s)\star\unfold(w')\star\unfold(v))\\
&=\fold(\unfold(w)\star\unfold(v)).
\end{align*}
This completes the proof.
\begin{comment}
We use induction on the length of $w$, with the base case being when $w$ is a simple transposition. In the case $w = s_i$ with $i < n$, the result follows since $w(a) > w(b)$ if and only if $-w(a) < -w(b)$: the ordering between $a$-th and $b$-th element in $w$ is the same as the ordering between the $(2n-b)$-th and $(2n-a)$-th element in the unfolding of $w$. In the case $w = s_n$, the result follows since $s_n * w = w$ if and only if we had $-n$ in the signed permutation $w$, which again happens if and only if we have $-n$ coming before $n$ in the unfolding of $w$.
Now assume for sake of induction we have the claim for signed permutations of lesser length than $w$. We write $w = s_i * w'$, and from the induction hypothesis we get:
$$ (s_i w') \star v = s_i \star (w' \star v) = s_i \star \fold(\unfold(w') \star \unfold(v)) = \fold(\unfold(s_i) * (\unfold(w') * \unfold(v))),$$
which in turn is same as $\fold(\unfold(w) * \unfold(v))$.
\end{comment}
\end{proof}

Next, we define a hopping operator for $B_n$ analogous to Definition \ref{def:hopping_typeA} for permutations.
\begin{definition}\label{def:hopping_typeB}
Let $t \in \pm [n]$ and $L$ an ordered subset $\pm [n]$ (without repetition).  The \newword{hopping operator}
$$h_{t,L}:B_n\rightarrow B_n$$ acts on a signed permutation $w$ by the following algorithm:  Scan to the right (within the unfolding of $w$) of $t$ and look for the element furthest to the right in $L$ that is greater than $t$. If it exists, say $q$, then swap $t$ and $q$ and also swap $-t$ with $-q$ (unless $t = -q$).  Replace $w$ with the resulting unfolded signed permutation and repeat. The algorithm ends when there are no elements of $L$ within $w$ to the right of $t$.
\end{definition}

For example, let $w=[2,3,5,-1,4]$ with $t=1$ and $L=[-2,-3,4]$.  We calculate the hopping operator $h_{1,[-2,-3,4]}(w)$. First we unfold $w$, which gives $$\unfold(w)=[2,3,5,-1,4,-4,1,-5,-3,-2].$$ To the right of $1$ we have $[-5,-3,-2]$. We first swap $1$ with $-3$, since $-3$ is the rightmost element of $L$ that exists here. This gives us $[2,-1,5,3,4,-4,-3,-5,1,-2]$. After that, we again scan to the right we have $[-2]$. Then we swap $1$ with $-2$, to get $[-1,2,5,3,4,-4,-3,-5,-2,1]$. So the signed permutation we end up with is $[-1,2,5,3,4]$.

%\begin{definition}
%For $t \in [n]$ and $L$ an ordered list of some elements of $\pm [n]$ (without repetition) the \newword{(left) hopping operator} $h_{\leftarrow t,L}$ acts on a signed permutation $w$ to yield the signed permutation obtained by the following algorithm: Scan to the left (within the unfolding of $w$) of $t$ and look for the element furthest to the left in $L$. If it exists, say $q$, swap $t$ and $q$ then also swap $-t$ with $-q$ unless $t = -q$, then replace $w$ with the resulting permutation and repeat. The algorithm ends when there are no elements of $L$ within $w$ to the left of $t$. Similarly the \newword{(right) hopping operator} $h_{\rightarrow t,L}$ does the same process but scans to the right of $t$ within the unfolding of $w$.
%\end{definition}

Just like the type $A$ case, hopping operators satisfy a commuting relation with simple transpositions as in Lemma \ref{lem:b_simplehop}.  We omit the proof since it is analogous to that of Lemma \ref{lem:b_simplehop}.  Similar to the type $A$ case, we let $B_n$ act on sub-lists of $\pm[n]$ via the corresponding signed permutation.

%s_i(L) definition?

\begin{lemma}\label{lem:typeB_simple_commute}
\label{lem:b_transhop}
Let $w\in B_n$.  For any $i \geq t$ and ordered subset $L\subseteq \pm[n]$ we have
$$s_i\cdot h_{t,L}(w) = h_{t,s_i(L)} (s_i\cdot w).$$
\end{lemma}

Recall that in the proof of Theorem \ref{thm:main}, we defined $\cs_{a,b}:=s_a \cdots s_{a+b-1}$ and used the fact that any permutation naturally decomposes into a product of $\cs_{a,b}$'s (see Equation \eqref{eqn:type_A_stringcomposition}).  For the type $B_n$ case, we will define the analogous product of simple generators
$$\cs^B_{a,b}:=s_a\cdots s_{a+b-1}$$
where for any $j>1$, we set $s_{n+j}:=s_{n-j}$.  Note that if $a\leq n$, then $1\leq b\leq 2n-a$.  For example, in $B_7$, we have
$$\cs^B_{5,6} = s_5 s_6 s_7 s_8 s_9 s_{10} = s_5 s_6 s_7 s_6 s_5 s_4.$$
As a signed permutation, the product $\cs^B_{a,b}$ corresponds to unfolding the identity permutation $[1,2,\ldots,n]$ and shifting $a$ to the right by $b$ positions, then placing $-a$ in the mirrored position.  For example, in $B_7$ we have $$\cs^B_{5,6} =[1,2,3,-5,4,6,7,-7,-6,-4,5,-3,-2,-1]=[1,2,3,-5,4,6,7].$$
An immediate corollary of Lemma \ref{lem:typeB_simple_commute} is the following.

%In the type $A$-case the building blocks we used were $C_{a,b}$'s: given any permutation, we could break it down into a product of $C_{a,b}$'s, and we could %describe how these blocks interact with the hopping operators.
 %

%For the type $B_n$ case we will do something similar. Let $C_{a,b}$ again denote $s_a \dots s_{a+b-1}$, but we are going to allow $a+b-1 > n$ and identify $s_{n+j}$ %with $s_{n-j}$ for $1 \leq j < n$. For example in $B_7$, we have $C_{5,6} = s_5 s_6 s_7 s_8 s_9 s_{10} = s_5 s_6 s_7 s_6 s_5 s_4$. Notice that $C_{a,b}$ in terms of %a signed permutation, is obtained from the identity permutation $[1,2,\ldots,n]$, unfolding it, then shift the entry $a$ to the right by $b$ positions (after that %place $-a$ in the flipped spot of $a$). For example, in $B_7$ we have $C_{5,6} = [1,2,3,-5,4,6,7]$.

\begin{corollary}
\label{cor:b_shcomm}
Let $w\in B_n$ with $a\leq n$ and $1\leq b\leq 2n-a$.  For any $t \leq a$ and ordered subset $L\subseteq \pm[n]$, we have
$$\cs^B_{a,b}\cdot h_{t,L}(w) = h_{t,\cs^B_{a,b}(L)}( \cs^B_{a,b}\cdot w).$$
%For any $t \leq i<j,L$ we have $s_i s_{i+1} \cdots s_j h_{t,L} = h_{t,s_is_{i+1}\cdots s_j(L)} s_i s_{i+1} \cdots s_j$.
\end{corollary}

%Similarly, when we are dealing with the elements of $\pm [n]$, we identify $-j$ with $n+j$: reason of doing this will be clear in the next paragraph.

As in the type $A$ case, the Demazure product with $\cs^B_{a,b}$ can be described using the usual composition product on signed permutations and the hopping operator given in Definition \ref{def:hopping_typeB}. Recall we identified $s_{n+j}$ with $s_{n-j}$ for $1 \leq j < n$. Similarly we use $n+j$ to denote $-(n+1-j)$ for $1 \leq j < n$ when we are dealing with elements of $\pm [n]$.

\begin{lemma}
\label{lem:b_simplehop}
Let $v \in B_n$. For any $i < n$, we have
$$s_i \star v = h_{i,[i+1]} (s_iv)$$ and
$$s_n \star v = h_{n,[-n]} (s_nv).$$
\end{lemma}
\begin{proof}  In this proof, let $h^A_{i,L}$ denote the hopping operator given in Definition \ref{def:hopping_typeA} acting on the permutation group $S_{2n}$ and $h^B_{i,L}$ denote the hopping operator given in Definition \ref{def:hopping_typeB} acting on $B_n\subseteq S_{2n}$.
If $i < n$, then $s_i = s_{-i}'s_i'$ and by Lemma \ref{lem:b_unfold}, we have
$$\unfold(s_i \star v) = s_{-i}'s_i' \star \unfold(v).$$
Proposition~\ref{prop:cyclehop} implies $$(s_{-i}'{s_i}') \star \unfold(v) = h^A_{-(i+1),[-i]} h^A_{i,[i+1]} s_{-i}'{s_i}' \unfold(v).$$
Note that swapping $i$ with $i+1$ in $h^A_{i,[i+1]}$ mirrors swapping $-(i+1)$ with $-i$ in $h^A_{-(i+1),[-i]}$. Hence Lemma~\ref{lem:b_unfold} implies
$$\fold((s_{-i}'{s_i}') \star \unfold(v))=h^B_{i,[i+1]}(s_iv).$$
In the case that $i = n$, note that $s_n \star v = v$ if the sign of $n$ in $v$ is negative and $s_n \star v = s_nv$  otherwise. From this it follows that $s_n \star v = h_{n,[-n]} (s_nv)$.
%The claim follows from Proposition~\ref{prop:cyclehop} and Lemma~\ref{lem:b_unfold}.
\end{proof}

\begin{proposition}
\label{prop:b_cyclehop}
Let $v \in B_n$. For any $i \leq j \leq 2n-1$, we have $$(s_i s_{i+1} \cdots s_j) \star v = h_{i,[i+1,\ldots,j+1]} (s_i s_{i+1} \cdots s_j v).$$
\end{proposition}
\begin{proof}
First, if $i\leq j < n$, we focus on how $s_i's_{i+1}'\dots s_j'$ interacts with $v$ since how  $s_{-i}'s_{-(i+1)}'\dots s_{-j}'$ interact with $v$ will mirror that. Then the proposition follows from Proposition~\ref{prop:cyclehop} and Lemma~\ref{lem:b_unfold}. Second, if $n<i\leq j$, we now focus on how $s_{-i}'s_{-(i+1)}'\dots s_{-j}'$ interacts with $v$ since how $s_i's_{i+1}'\dots s_j'$ interact with $v$ will mirror that.  Note that $-i < -(i+1) < \dots < -j$ is an increasing order and hence Proposition~\ref{prop:cyclehop} and Lemma~\ref{lem:b_unfold} imply $$(s_i s_{i-1} \dots s_j) \star v = h_{-(i+1),[-i,\ldots,-j]} (s_i s_{i-1} \dots s_j v).$$ This proves the proposition when $i > n$.

\smallskip

Now suppose that $i \leq n\leq j$.  We proceed by induction on $n-i$.  First, when $n-i = -1$, the proposition follows from the above case. Now suppose that $n-i\geq 0$ and suppose for the sake of induction that we have the proposition is true for all $i'$ such that $n-i'<n-i$.   We start by analyzing the expression $s_i \star(( s_{i+1} \cdots s_j) \star v)$. From the induction hypothesis, we have that
$$(s_{i+1} \cdots s_j) \star v = h_{i+1,[i+2,\ldots,n+j]} (s_{i+1} \cdots s_{n+j} v).$$
By Lemma~\ref{lem:b_simplehop}, it suffices to analyze the operator $h_{i,[i+1]} s_i h_{i+1,[i+2,\ldots,2n]}$.  Then Lemma~\ref{lem:b_transhop} implies $$s_i h_{i+1,[i+2,\ldots,2n]} s_i = h_{i,[i+2,\ldots,2n]}$$ and hence $h_{i,[i+1]}h_{i,[i+2,\ldots,2n]} = h_{i,[i+1,\ldots,2n]}$.  This proves the proposition.

%Now in the case $i \leq n$, we start from $s_n \star( (s_{n-1} \cdots s_j) \star v)$. Since we already know that $s_{n-1} \dots s_j \star v = h_{-n,[-(n-1),\ldots,-j]} (s_{n-1} \cdots s_j v)$ from above, combined with Lemma~\ref{lem:b_simplehop} the previous expression is $h_{n,[-n]} s_n h_{-n,[-(n-1),\ldots,-1]}$. Since $s_n h_{-n,[-(n-1),\ldots,-1} s_n = h_{n,[-(n-1),\ldots,-1]}$ from Lemma~\ref{lem:b_transhop}, we get the desired claim in this case.

%Now in the case $i \leq n$ and $j \geq n$, we use induction on $n-i$. The base case, when $n=i$ we already have shown. We start from $s_i \star ((s_{i+1} \cdots s_j) \star v)$. Since we already know that $(s_{i+1} \cdots s_j) \star v = h_{}
\end{proof}

We now give an analogue of Definition \ref{def:typeA_left_subword} for signed permutations.  For any $w\in B_n$ and $i>0$, define $w \nwarrow i$ to be the subword of $\unfold(w)$ obtained be restricting to numbers to the left of $i$ that are either greater than $i$, or less or equal to $-i$. For example, if $w = [-5,3,1,-2,4]$, then
$$\unfold(w)=[-5,3,1,-2,4,-4, 2, -1, -3, 5].$$ In this case we have $w\nwarrow 1 = [-5,3]$ and $w \nwarrow 2 = [-5,3,-2,4,-4]$.

\begin{theorem}\label{thm:b_main}
For any $w,v\in B_n$, we have $w \star v = h_{n-1,w_{\nwarrow n-1}} h_{2,w_{\nwarrow 2}} h_{1,w_{\nwarrow 1}}(wv)$
\end{theorem}
\begin{proof}
The signed permutation $w$ has a unique decomposition
$$w=\cs^B_{n,j_n}\cdots \cs^B_{2,j_2} \cs^B_{1,j_{1}}$$
for some ``inversion" sequence $(j_1,\ldots,j_n)$.
The theorem now follows an analogue of the proof of Theorem~\ref{thm:main} where we use Proposition~\ref{prop:b_cyclehop} and Corollary~\ref{cor:b_shcomm} instead of Proposition~\ref{prop:cyclehop} and Corollary~\ref{cor:shcomm}.
\end{proof}

The decomposition of signed permutations as a product of $\cs^B_{a,b}$'s stated in the proof above is well known, but due to a lack of reference, we give a brief construction.
The interpretation of $\cs^B_{a,b}$ is that it adjusts the position and sign of the number $a$ in the unfolding of a given signed permutation. Specifically, let $w\in B_n$ and starting from the identity, we use $\cs^B_{1,j_1}$ to move 1 to it's position within $\unfold(w)$.  Note that as 1 move, we also move $-1$ accordingly to stay in $B_n$.  Here, the number $j_1$ denotes the number of positions 1 needs to move.  We then use $\cs^B_{2,j_2}$ to move 2 (and $-2$) to it's position within $\unfold(w)$.  Repeating this process with all numbers $1,2,\ldots, n$ constructs the signed permutation $w$. For example let $w = [-5,3,1,-2,4]$ with $[-5,3,1,-2,4,-4,2,-1,-3,5]$ as its unfolding. The decomposition we get is $$w = \cs^B_{5,9} \cs^B_{4,4} \cs^B_{3,1} \cs^B_{2,6}\cs^B_{1,2}.$$

Next, we give an example of computing the Demazure product on signed permutations using Theorem \ref{thm:b_main}.  Let $w=[-5,3,1,-2,4]$ and $v=[-4,2,-1,-3,5]$ in $B_5$.  We compute $w\star v.$  First note the usual product $wv=[2, 3, 5, -1, 4].$  We apply the sequence of hopping operators
$$h_{5,[-5]}h_{4,[-5]}h_{3,[-5]}h_{2,[-5,3,-2,4,-4]}h_{1,[-5,3]}$$
to $wv$ giving:

\begin{align*}
& [2,3,5,-1,4,-4,1,-5,-3,-2] \xrightarrow[h_{1,[-5,3]}]{} [2,3,-1,5,4,-4,-5,1,-3,-2]  \xrightarrow[h_{2,[-5,3,-2,4,-4]}]{} \\
& [-2,3,-1,5,-4,4,-5,1,-3,2] \xrightarrow[h_{3,[-5]}]{} [-2,-5,-1,-3,-4,4,3,1,5,2]  \xrightarrow[h_{4,[-5]}]{} \\
& [-2,-5,-1,-3,-4,4,3,1,5,2] \xrightarrow[h_{5,[-5]}]{} [-2,-5,-1,-3,-4,4,3,1,5,2]
\end{align*}
Theorem \ref{thm:b_main} implies that the Demazure product $w\star v=[-2, -5, -1, -3, -4]$.
We conclude this section with some questions.

%\begin{align*}
%7142563 \xrightarrow[h_{1,[6,5,4]}]{}  7452613 \xrightarrow[h_{2,[6,5,4,7]}]{} & 7456213\\  %
%& \xrightarrow[h_{3,[6,5,4,7]}]{}  7456213  \xrightarrow[h_{4,[6,5]}]{}  7564213 \xrightarrow[h_{4,[6,5]}]%{} 7654213.
%\end{align*}

\begin{question}
Can the Demazure product in type $D$ be similarly described as the type $B$ case?
\end{question}

\begin{question}
In \cite{BW}, Billey and Weaver give a ``one-line notation" algorithm to compute the maximal element in the intersection of a lower interval with an arbitrary coset of a maximal parabolic subgroup in type $A$.  In \cite{OR}, there is an alternate algorithm to compute the maximal element using the Demazure product (this second formula is for any Coxeter group and parabolic subgroup).  Is there a way to apply Theorem \ref{thm:main} to recover the algorithm in \cite{BW}?  If so, is there a generalization of the algorithm in \cite{BW} to the case where the parabolic subgroup is not maximal? or to the case of signed permutations?  We remark that the existence of such a maximal element for any Coxeter group $W$ and parabolic subgroup $W_J$ was established in \cite{Mazorchuk-Mrden20}.
\end{question}

\subsection*{Acknowledgments}
This research was primarily conducted during the 2022 Honors Summer Math Camp at Texas State University. The authors gratefully acknowledge the support from the camp and also thank Texas State University for providing support and a great working environment.  ER was supported by a grant from the Simons Foundation 941273.

\bibliography{bruhat}
\bibliographystyle{siam}

\end{document}